\documentclass[11pt]{article}
\usepackage{graphicx,amsthm,lineno}

\title{Ramsey numbers for degree monotone paths}
\date{}
\begin{document}
\newtheorem{theorem}{Theorem}[section]
\newtheorem{definition}{Definition}[section]
\newtheorem{proposition}[theorem]{Proposition}
\newtheorem{corollary}[theorem]{Corollary}
\newtheorem{problem}[theorem]{Problem}
\newtheorem{conjecture}[theorem]{Conjecture}
\newtheorem{lemma}[theorem]{Lemma}
\newcommand*\cartprod{\mbox{ } \Box \mbox{ }}
\newtheoremstyle{break}% name
  {}%         Space above, empty = `usual value'
  {}%         Space below
  {\itshape}% Body font
  {}%         Indent amount (empty = no indent, \parindent = para indent)
  {\bfseries}% Thm head font
  {.}%        Punctuation after thm head
  {\newline}% Space after thm head: \newline = linebreak
  {}%         Thm head spec

\theoremstyle{break}

\newtheorem{propskip}[theorem]{Proposition}
\DeclareGraphicsExtensions{.pdf,.png,.jpg}
\author{Yair Caro \\ Department of Mathematics\\ University of Haifa-Oranim \\ Israel \and Raphael Yuster \\ Department of Mathematics\\ University of Haifa \\ Israel  \and Christina Zarb \\Department of Mathematics \\University of Malta \\Malta}

\maketitle

\begin{abstract}
A path $v_1,v_2,\ldots,v_m$ in a graph $G$ is \emph{degree-monotone} if  $deg(v_1) \leq deg(v_2) \leq \cdots \leq deg(v_m)$ where $deg(v_i)$ is the degree of $v_i$
in $G$. Longest degree-monotone paths have been studied in several recent papers. Here we consider the Ramsey type problem for degree monotone paths.
Denote by $M_k(m)$ the minimum number $M$ such that for all $n \geq M$, in any $k$-edge coloring of $K_n$ there is some $1\leq j \leq k$ such that the graph
formed by the edges colored $j$ has a degree-monotone path of order $m$. We prove several nontrivial upper and lower bounds for $M_k(m)$.
\end{abstract}

\section{Introduction}

A path $v_1,v_2,\ldots,v_m$ in a graph $G$ is \emph{degree-monotone} if  $deg(v_1) \leq deg(v_2) \leq \cdots \leq deg(v_m)$ where $deg(v_i)$ is the degree of $v_i$
in $G$. The maximum order over all degree-monotone paths in $G$ is denoted by $mp(G)$. The study of properties related to the longest degree-monotone path was explicitly
suggested by Deering et al. \cite{updownhilldom}. A more general monotone path problem was suggested long ago by Chvatal and Koml\'os \cite{chvatal1971some}, who
related oriented graphs and oriented paths to various monotonicity problems, motivated by the famous Erd{\H{o}}s-Szekeres Theorem  \cite{eliavs2013higher,erdos1935combinatorial} on monotone
sub-sequences, and by the Gallai-Roy Theorem (see \cite{dougwest}). Caro et al. \cite{dmpclz,dmp3} studied $mp(G)$ and other related parameters.

One important observation which is immediate from the Gallai-Roy theorem is that $mp(G) \geq \chi(G)$.
Indeed, if we orient an edge from a low degree vertex to a high degree vertex (breaking ties arbitrarily), then a directed path in the resulting oriented graph corresponds to a
degree-monotone path in the original undirected graph, and the Gallai-Roy Theorem asserts that in any orientation, the order of a longest directed path is at least as large as the chromatic number.

Our goal here is to study the Ramsey type problem for degree-monotone paths.
Denote by $M = M(m_1,m_2,\ldots,m_k)$ the minimum number $M$ such that for all $n \geq M$, in any $k$-edge coloring of $K_n$,  for some $j$ where $1\leq j \leq k$,  the spanning monochromatic graph
$G_ j$ formed by the edges colored $j$ satisfies $mp(G_j) \geq m_j$.
In the diagonal case $m=m_1= \cdots=m_k$, we write $M_k(m)$.  We refer to a monochromatic degree-monotone path in this context as an \emph{mdm-path} for short.
We will always assume that $k \ge 2$ and $m \ge 3$ to avoid the trivial cases.
 
One should observe a subtlety in the definitions of $M_k(m)$ (as well as $M(m_1,\ldots,m_k)$). It is not clear that if $n$ is the smallest integer for which $K_n$ satisfies the stated property,
then $M_k(m)=n$. This is because being true for $n$, does not a priori imply it for $n+1$ as the property of having a degree-monotone path is not hereditary.
Indeed, even the addition of a vertex $v$, and coloring all the edges joining it to the other vertices using $k$ colors might destroy some mdm-paths that existed prior to adding $v$.
Hence the requirement in the definition that $M$ is the smallest integer such that for {\em all} $n \ge M$ the stated property holds, is important.
These sort of Ramsey-degree problems (with the related subtle monotonicity problem just mentioned) originated in some papers by Albertson \cite{albertson1992turan,albertson1994people} and
Albertson and Berman \cite{albertson1991ramsey}, and were further developed shortly afterward by Chen and Schelp \cite{CPCchen} and Erd\H{o}s et al. \cite{Erdos1993183}. 
We mention the following interesting result that appeared in \cite{Erdos1993183}.
\begin{theorem}
In any $2$-coloring of the edges of $K_n$, where $n \geq R(m,m)$, there is a monochromatic copy of $K_m$  with vertices $v_1,\ldots,v_m$  such that  in the host monochromatic graph $G$,
$$
\max\{ deg(v_i)  : i=1, \ldots, m)\}  -  \min\{  deg(v_i)  :i = 1, \ldots, m\} \leq  R(m,m)-2\;,
$$
and this is sharp for $n \geq 4(r-1)(r-2)$ where $r = R(m,m)$.
\end{theorem}
 
Having all these facts in mind we are now ready to state our first main result, which provides general upper and lower bounds for $M_k(m)$.
\begin{theorem}\label{t:1}
$$
\frac{ (m-1)^k}{2}  +\frac{m-1}{2 }+1  \le M_k(m) \le (m-1)^k +1\;.
$$ 
In fact, more generally, $M(m_1,\ldots,m_k)  \leq \prod_{i=1}^{k}(m_i -1) +1$.
\end{theorem}
\noindent
Notice that the upper and lower bounds for $M_k(m)$ differ by a factor smaller than $2$.

As usual in most Ramsey type problems, proving tighter bounds, or even computing exact small values, turns out to be a difficult task already in the first, and perhaps most interesting, case of
paths of order $3$, namely $M_k(3)$. This case can also be interpreted as requiring that the degree of every vertex of a graph with no isolated edges is a local extremum (either strictly smaller than
the degree of all its neighbors or strictly larger than the degree of all its neighbors).
Observe that Theorem \ref{t:1} gives $2^{k-1}+2 \leq M_k(3) \leq 2^k+1$. Our next theorem improves both upper and lower bounds.
\begin{theorem}\label{t:2}
$M_2(3) = 4$, $M_3(3) = 8$  and $\frac{3}{4}2^k +2 \le M_k(3)\leq 2^k -1$ for $k \geq 4$.
\end{theorem}
\noindent
We note that while the upper bound is only a mild improvement over the one provided by Theorem \ref{t:1}, its proof turns out to be somewhat involved.

The first off-diagonal nontrivial case is $M(3,m)$ for which we prove:
\begin{theorem}\label{t:3}
$M(3,m)=2(m-1)$.
\end{theorem}

In the rest of this paper we prove the general bounds in Section 2, the more involved tighter bounds for paths of order $3$ are proved in Section 3, and the proof of Theorem \ref{t:3} appears in
Section 4. The final section contains some specific open problems. Our notation follows that of \cite{dougwest}, and will otherwise be introduced when it first appears.

\section{General upper and lower bounds}

In this section we prove Theorem \ref{t:1}.
The upper bound in Theorem \ref{t:1} is a consequence of the following result proved independently by 
Gy\'arf\'as and Lehel \cite{GL-1973}, Bermond \cite{bermond-1974}, and Chvatal \cite{chvatal-1972}. They used an observation of Zykov \cite{zykov1949some}
that states that in any edge coloring of a complete graph with more than $\prod_{i=1}^k (m_i-1)$ vertices with $k$ colors, there is a color $i$ that induces a graph
whose chromatic number is at least $m_i$, together with the Gallai-Roy Theorem to deduce:
\begin{lemma}\label{l:1}
In any $k$-coloring of the edges of a tournament on more than $\prod_{i=1}^k (m_i-1)$ vertices, there is a directed path of order $m_i$, all of whose edges are colored $i$.
The bound $\prod_{i=1}^k (m_i-1)$ is tight. Furthermore, in any extremal example, the chromatic number of the graph whose edges are colored with color $i$
is $m_i-1$ and any proper $(m_i-1)$-vertex coloring of it is equitable (all vertex classes have equal size).
\end{lemma}
The upper bound in Theorem \ref{t:1} is a consequence of Lemma \ref{l:1} obtained as follows. Consider a coloring of $K_n$ with $k$ colors, and in each colored graph $G_i$,
orient an edge $uv$ colored with $i$ from $u$ to $v$ if $deg_i(u) > deg_i(v)$ where $deg_i(x)$ is the degree of $x$ in $G_i$ (break ties arbitrarily).
We then obtain a coloring of a tournament with $k$ colors. Now, if $n > \prod_{i=1}^k (m_i-1)$, Lemma \ref{l:1} asserts that there is a monochromatic directed path of order $m_i$
all of whose edge are colored $i$. This path is, by construction, an mdm-path in $G_i$. This proves that $M(m_1,\ldots,m_k) \le 1+\prod_{i=1}^k (m_i-1)$
and in particular, $M_k(m) \le 1+(m-1)^k$. Observe that a construction showing tightness in Lemma \ref{l:1} is not necessarily relevant in our setting as it may not imply tightness for the
degree-monotone problem. Nevertheless, as the lemma states, it does imply that {\em if} the bound $1+\prod_{i=1}^k (m_i-1)$ is tight, then any extremal example on
$\prod_{i=1}^k (m_i-1)$ vertices must have that $mp(G_i)=m_i-1$ and that any $m_i-1$ coloring of $G_i$ is equitable.

\vspace{3mm}
The next lemma proves the lower bound in Theorem \ref{t:1}.
\begin{lemma}\label{l:lower}
$M_k(m) \geq \frac{(m-1)^k}{2}+ \frac{m-1}{2} +1$.
\end{lemma}
\begin{proof}
We will prove the stronger claim that for each integer $n$ of the form $\frac{(m-1)^k}{2}+ \frac{m-1}{2}-t$ for $t=0,\ldots,m-1$,
there is an edge coloring of $K_n$ with $k$ colors and with no mdm-path of order $m$.

We proceed by induction on $k$, starting with $k=2$.
Let $X_1,\ldots,X_{m-1}$ be sets such that $|X_j|=j$ for $j=1,\ldots,m-1$.
Form a complete graph on $V=\cup_{j=1}^{m-1}X_j$ by coloring an edge with both endpoints in the same set with color $1$ and an edge with endpoints in distinct sets with color $2$.
As the color $1$ induces a graph $G_1$ whose components are cliques of order at most $m-1$, there is no path on $m$ vertices in $G_1$.
For the color $2$, observe that any path on $m$ vertices in the graph $G_2$ must contain two non-consecutive vertices from the same set $X_j$ for some $j$.
But any two vertices in $X_j$ have the same degree in $G_2$ and this degree is distinct from the degree in $G_2$ of any vertex not in $X_j$.
Hence there is no mdm-path of order $m$ in $G_2$. As the number of vertices is $|V|=\sum_{j=1}^{m-1} j = m(m-1)/2$, the claim holds for $k=2$ with $t=0$.
However, notice that the same argument holds if we take a smaller union $V \setminus X_t$ for $t=1,\ldots,m-1$ (just take the same coloring and omit $X_t$).
Hence, the claim holds for $k=2$ and $t=0,\ldots,m-1$.

Now assume we have proved that there are complete graphs on $\frac{(m-1)^k}{2}+ \frac{m-1}{2}-t$ vertices for $t=0,\ldots,m-1$, and a $k$-edge coloring of each of them
with no mdm-path of order $m$. We prove for $k+1$.
Denote such colored complete graphs by $X_0, \ldots, X_{m-1}$ where $X_t$ has $\frac{(m-1)^k}{2}+ \frac{m-1}{2}-t$ vertices.
Let $Y_t$ be the complete graph obtained by taking the disjoint union of $X_0,X_1,\ldots,X_{t-1},X_{t+1},\ldots,X_{m-1}$ (using the existing $k$-coloring in each component)
and color any two vertices with endpoints in distinct $X_j$ with color $k+1$.
By induction, there is no mdm-path of order $m$ on colors $1,\ldots,k$ and there is also no mdm-path of order $m$ on color $k+1$ since 
any path on $m$ vertices in the graph $G_{k+1}$ (the subgraph of $Y_t$ on the edges colored $k+1$) must contain two non-consecutive vertices from the same subgraph $X_j$ for some $j$.
But any two vertices in the same $X_j$ have the same degree in $G_{k+1}$ and this degree is distinct from the degree in $G_{k+1}$ of any vertex not in $X_j$.
Hence $Y_t$ has no mdm-path of order $m$.
Now notice that
\begin{eqnarray*}
|V(Y_{m-1-t})| & = & \left(\sum_{s=0}^{m-1} (\frac{(m-1)^k}{2}+ \frac{m-1}{2}-s)\right)  - \left(\frac{(m-1)^k}{2}+ \frac{m-1}{2}-(m-1-t)\right)\\
& = &\frac{(m-1)^{k+1}}{2}+ \frac{m-1}{2}-t
\end{eqnarray*}
proving the induction step for $k+1$ and $t=0,\ldots,m-1$.
\end{proof}

\noindent
\textbf{Remark 1}: we observe that once we have $m$ consecutive integers $t,t-1,t-2,\ldots,t-m+1$ for which it is possible to $k$-color the edges of $K_{t-j}$, $j=0,\ldots,m-1$ without an mdm-path
 of order $m$, then we can $(k+1)$-color the graph $K_{q-j}$ for $j=0,\ldots,m-1$ without an mdm-path of order $m$, where $q=(m-1)t-\frac{(m-1)(m-2)}{2}$ and the process can be continued.
So whenever we have an improvement of the basic lower bound, we can carry over this new better bound. We shall use this to prove the lower bound for $M_k(3)$ obtained in Theorem \ref{t:2}.

\section{Paths of order $3$}

\subsection{A structural property}
We consider certain conditions imposed on the degrees of bipartite graphs, and then use the structural properties of these  bipartite graphs when such graphs exists, and the non-existence of such
graphs otherwise, to prove the upper bound in Theorem \ref{t:2}.

A bipartite graph with bipartition $V=A \cup B$ is said to be \emph{illusive} if:
\begin{itemize}
\item{$|A| > |B|$, $A$ has no isolated vertices, and for every vertex $v \in A$, $deg(v) \geq deg(u)$ for all vertices $u \in N(v)$, or}
\item{$|A|=|B|$, $A$ has no isolated vertices, and for every vertex $v \in A$, $deg(v) \geq deg(u)$ for all vertices $u \in N(v)$.  Furthermore, there exists $v \in A$ such that
$deg(v)>deg(u)$, for some $u \in N(v)$.}
\end{itemize}

\begin{lemma}\label{l:illusive}
Illusive graphs do not exist.
\end{lemma}
\begin{proof}
We consider first the case $|A| > |B|$ and assume by contradiction that $G$ is a minimum counter-example, namely $G$ is an illusive graph with a minimum number of vertices, such that $|A|>|B|$.
Let $|A|=p$ and $|B|=q$, $p>q$. Let us order the vertices $v_1,v_2,\cdots, v_p \in A$ such that $deg(v_1) \geq \cdots \geq deg(v_p)$.  

Consider the set of vertices $S=\{v_1,\ldots, v_q\} \subset A$.  Now if $S$ has a matching to $B$ in $G$, then we can arrange the vertices $u_1,\ldots,u_q \in B$ such that $v_i$ is adjacent to $u_i$
for $i=1,\ldots, q$, and  clearly $deg(v_i) \geq deg(u_i)$ as they are neighbors in $G$ and no vertex in $S$ lost any neighbor.

Now as $deg(v_p) \geq 1$, we clearly have \[|E(G)|= \sum_{i=1}^{i=p} deg(v_i) > \sum_{i=1}^{i=q} deg(u_i) = |E(G)|,\] a contradiction.  Hence, there is no matching between $S$ and $B$.  Now by
Hall's Theorem \cite{dougwest}, there exists $Q \subset S$ such that $|N(Q)| < |Q|$.  Consider the subgraph $H$ of $G$ induced by $Q \cup N(Q)$.  Clearly $|V(H)| < |V(G)|$.  $H$ is an illusive graph
because any vertex in $Q$ has all its neighbors in $N(Q)$, and only vertices in $N(Q)$ can lose some of its neighbors which are not in $Q$. Also, since no vertex in $A$ is isolated, it follows that
no vertex in $Q$ is isolated as all the neighbors of the vertices in $Q$ are in $N(Q)$.
Hence, $H$ is an illusive graph with $|V(H)| < |V(G)|$, a contradiction to the minimality of $G$.

Now consider  the case $|A| = |B|$ and assume by contradiction that $G$ is a minimum counter-example, namely $G$ is an  illusive graph with a minimum number of vertices, such that $|A|=|B|=p$ and
furthermore from all such graphs with $|A|=|B|=p$,  let $G$ have the minimum number of edges. Let us order the vertices $\{v_1,\ldots,v_p\}$ in $A$ such that $deg(v_1) \geq \cdots \geq deg(v_p)$.

If there is a matching in $G$ from $A$ to $B$, then we can rearrange the vertices $\{u_1,\ldots,u_p\}$ of $B$ such that every vertex $v_i$ is adjacent to $u_i$ for $i=1, \ldots, p$, and clearly
$deg(v_i) \geq deg(u_i)$ as they are neighbors in $G$.

Let us delete the edges $(v_i,u_i)$ $i=1, \ldots, p$ to obtain $G^*$ such that $V(G^*)=V(G)$ but $|E(G^*)| < |E(G)|$.  Every vertex in $G^*$ has degree one less than that in $G$, so we still have
$deg(v_i) \geq deg(u_i)$ for every $v_i \in A$ and every $u_i \in N(v_i)$.

Now if some vertex in $A$ in $G^*$ has degree $0$, then its neighbor in the matching in $B$ must also have degree $0$. We consider the following two cases:
\begin{enumerate}
\item
In $G^*$ there are more vertices of degree $0$ in $B$ than in $A$.  Let us delete all the vertices of degree 0 from $A$ and $B$ to get $A^*$ and $B^*$, and $H=A^* \cup B^*$, with $|V(A^*)|>|V(B^*)|$. 
But then $H$ is illusive of  the type which we proved to be impossible in the first part of the proof.
\item
There are exactly the same number of vertices of degree 0 in $A$ and in $B$ (possibly no isolated vertices at all).
Let us delete all the vertices of degree 0 from $A$ and $B$ to get $A^*$ and $B^*$, and $H=A^* \cup B^*$ a subgraph of $G^*$, with
$|V(A^*)|=|V(B^*)|$.  Recall that there exists a vertex $v \in A$ such that there is a vertex $u \in N(v)$ with $deg(v)>deg(u)$.  Now if in the matching in $G$, $v$ is matched with $u$, then in $G$
\[|E(G)|= \sum_{i=1}^{i=p} deg(v_i) > \sum_{i=1}^{i=p} deg(u_i) = |E(G)|,\] which is not possible. Hence $v$ is not matched to $u$.  But then $deg(v) \geq 2$ and $v$  is still connected to $u$ in
$H$, and hence in $H$, $deg(v)>deg(u)$ and $u$ and $v$ are adjacent, which implies that $H$ is a smaller illusive graph, a contradiction.
\end{enumerate}

Finally let us assume that there is no matching between $A$ and $B$ in $G$. Again, by Hall's Theorem, there exists $Q \subset A$ such that $|N(Q)| < |Q|$.  Consider the subgraph $H$ of $G$ induced
by $Q \cup N(Q)$.  Clearly $|V(H)| < |V(G)|$.  $H$ is an illusive graph because any vertex in $Q$ has all its neighbors in $N(Q)$, and only vertices in $N(Q)$ can lose some of its neighbors which are
not in $Q$.  Also, since no vertex in $A$ is isolated, it follows that no vertex in $Q$ is isolated as all the neighbors of the vertices in $Q$ are in $N(Q)$.
Hence, $H$ is an illusive graph of the type proved impossible in the first part of this proof.  Hence illusive graphs do not exist.
\end{proof}

\noindent
An immediate consequence  of Lemma \ref{l:illusive} is:
\begin{corollary}
Let $G$ be a connected  bipartite graph with bipartition $V = A \cup B$  such that  $|A| \geq |B|$  and for every vertex $v  \in A$,  $deg(v) \geq deg(u)$ for every $u\in  N(v)$.   Then $|A| = |B|$
and $G$ is regular.
\end{corollary}
\begin{lemma}\label{l:illusive2}
Suppose $G$ is a bipartite graph with $V = A \cup B$  such that $|A| = k$ and $ |B| = k +1$, and such that for every vertex $v \in A$, $deg(v) \geq  1$ and $deg(v) > deg( u)$ for every
$u \in  N(v)$. Then $G = K_{k,k+1}$.
\end{lemma}
\begin{proof}
Let us order the vertices $ u_1,\ldots,u_{k+1}$ of $B$ in non-increasing order  so that $deg(u_1)
\geq \cdots \geq deg(u_{k+1})$.  Let  $B^* = B \backslash  u_{k+1}$.
Suppose first that $A$  has a perfect matching to $B^*$.  Then the vertices $ v_1,\ldots,v_k$  of $A$ can be ordered such that $v_i$ is adjacent to $ u_i$ and $deg(v_i) \geq deg( u_i)  +1$.
Counting edges in $G$  we get \[|E(G)| = \sum_{i=1}^k deg(v_i)    =   \sum_{i=1}^{k+1} deg( u_i ) \leq \left(\sum_{i=1}^{k} deg( v_i)  - 1 \right) + deg(u_{k+1})\]\[  =  \left(\sum_{i=1}^{k}
deg( v_i) \right) - k +  deg(u_{k+1})  = |E(G)| - k + deg( u_{k+1}).\]
Hence $deg(u_{k+1})$  must be equal $k$ (since $|A| = k$),  and since $u_{k+1}$ has minimum  degree in $B$ it forces all other vertices in $B$ to have degree $k$, and  hence $G = K_{k,k+1}$.
 
Hence suppose $A$ has no perfect matching to $B^*$.  Then by Hall's Theorem there is a subset $Q$ in $A$ such that $|N(Q)| < |Q| \leq |A| =k$ .
Consider the bipartite graph $H$  induced by the parts $Q$  and $N(Q) \cup  \{ u_{k+1}\}$.
Since the only possible neighbor of the vertices of $Q$ not in $B^*$ is $u_{k+1}$, it follows that  $Q$  and
$N(Q) \cup  \{ u_{k+1}\}$ induce a bipartite graph $H$ where the degrees of all vertices
in $Q$ are strictly larger than the degrees of their neighbors in $N(Q) \cup  \{ u_{k+1}\}$.
Since $|N(Q) \cup  \{ u_(k+1)\}| \le |Q|$,  $H$ is an illusive graph which doesn't exists.
\end{proof}
 
\subsection{Proof of Theorem \ref{t:2}}

We start with the following proposition that yields the upper bound $M_3(k) \le 2^k$. It will be useful to establish the small values $M_2(3)$ and $M_3(3)$.
\begin{proposition}\label{p:1}
$M_k(3) \leq 2^k$.
\end{proposition}
\begin{proof}
We already know that $M_k(3) \le 2^k+1$ from Theorem \ref{t:1}, so to establish the proposition it suffices to consider $k$-edge colorings of $K_{2^k}$.
Suppose, for contradiction, that we can color the edges of $K_{2^k}$ using $k$ colors such that there is no mdm-path of order $3$.
Let $G_j$ be the spanning graph whose edges are colored $j$ for $j=1,\ldots, k$. So by Lemma \ref{l:1} and the paragraph following it, our coloring is
an extremal example and thus for all $j=1,\ldots, k$ we have $\chi(G_j)=2$ and in every bipartition of $G_j$ both parts have the same order $2^{k-1}$.

Hence, each component of $G_j$ is a bipartite graph with bipartition $A,B$ where $|A|=|B|$.
Consider any such component which is not a $K_2$. Hence $|A|=|B| \ge 2$.
The degrees of the vertices in any path connecting a vertex from $A$ with another vertex from $A$ form
a sequence of integers with odd length and with no monotone subsequence of order $3$. As any two vertices of $A$ can be connected via a path, we have that either all vertices of $A$
have degree larger than all the degrees of their neighbors is $B$ or vice versa. Assume the former. Then this component is illusive, and by Lemma \ref{l:illusive}, this is impossible.

Hence all components of $G_j$ are $K_2$ and therefore all $G_j$ for $ j =1,\ldots,k$ are perfect matchings.
So we cover $K_{2^k}$ by $k$ matchings  each having precisely $2^{k-1}$ edges.
Hence
\[
k2^{k-1} = \frac{2^k(2^k - 1)}{2}
\]
and thus $k = 2^k -1$ which is impossible for $k \geq 2$, a contradiction.
\end{proof}

\begin{figure}[t!]
\centering
\includegraphics{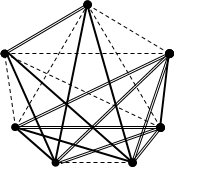}
\caption{$K_7$  decomposed into three copies of $K_{2,3} \cup K_2$} \label{diag1}
\end{figure}

\begin{corollary}\label{col:1}
$M_2(3)=4$ and $M_3(3)=8$.
\end{corollary}
\begin{proof}
By Proposition \ref{p:1} we have $M_2(3) \leq 4$.  Trivially, if we color $K_3$ using two colors we do not have an mdm-path of order $3$.  Hence $M_2(3)=4$.
Similarly, we know that $M_3(3) \leq 2^3=8$.  We can color the edges of $K_7$ in such a way that $G_i = K_{2,3} \cup K_2$ for $i=1,2,3$, as shown in Figure \ref{diag1}.
It is easy to see that $mp(K_{2,3} \cup K_2)=2$, and hence $M_3(3)= 8$. 
\end{proof}

As the sequence starts with $M_2(3)=2^2$ and $M_3(3)=2^3$ and since $M_k(3) \le 2^k$ one may wonder whether $M_k(3)=2^k$. The following lemma shows that this is not the case
already for $k \ge 4$. Somewhat surprisingly, the proof requires some effort.
\begin{lemma}\label{l:mk3}
For $k \geq 4$, $M_k(3) \leq 2^k-1$.
\end{lemma}
\begin{proof}
We already know that $M_k(3) \le 2^k$ so to establish the lemma it suffices to consider $k$-edge colorings of $K_{2^k-1}$.
Let $G_j$ be the spanning graph whose edges are colored $j$ for $j=1,\ldots, k$. If for some $j$, $G_j$ is not bipartite, then there is an
mdm-path of order $3$, so we may assume that each $G_j$ is bipartite.

We claim that in any bipartition of $G_j$, one side has $2^{k-1}$ vertices (and thus the other side has $2^{k-1}-1$ vertices).
Indeed, otherwise, one side would contain more than $2^{k-1}$ vertices, and induces an edge coloring with $k-1$ colors (all colors except $j$), so by Lemma \ref{l:1} and the paragraph following it,
there is an mdm-path of order $3$ in one of these colors.

So, each component of $G_j$ is a bipartite graph where the two sides have equal size, except precisely one component where the two sides differ in size by $1$.
(If there were more that one such component we could arrange a bipartition of $G_j$ into two sides whose sizes differ by more than $1$, and we have shown that this is impossible).
Now, by Lemmas \ref{l:illusive} and \ref{l:illusive2}, if there is no mdm-path of order $3$, then each balanced component must be a single edge, and the non-balanced component must be
$K_{b,b-1}$ for some integer $b$.

Hence $G_j$ is of the form $K_{b,b-1} \cup (2^{k-1}-b)K_2$. For each $G_j$, we call the $K_{b,b-1}$ component the {\em essential 
component} and the remaining matching on $2^{k-1}-b$ edges is the {\em non-essential part}.

Now, the number of edges of $K_{2^k-1}$ is $(2^k-1)(2^k-2)/2$, so the average number of edges in a colored graph is  $(2^k-1)(2^k-2)/2k=(2^k-1)(2^{k-1}-1)/k$. 
So, if we consider the largest colored graph, it is of the form $K_{a,a-1} \cup (2^{k-1}-a)K_2$ where we must have $a(a-1)+2^{k-1}-a \ge (2^k-1)(2^{k-1}-1)/k$.
Solving for $a$ we obtain that we must have
$$
a-1 \ge \left\lceil \sqrt{(2^{k-1}-1)(2^k-k-1)/k} \right\rceil\;.
$$
For example, if $k=5$ we must have $a \ge 10$.

Without loss of generality, let $G_k$ is the largest colored graph.
So, let us consider the essential component of $G_k$.  It is a complete bipartite graph with sides $A,B$ with $|A|=a$ and $|B|=a-1$.
Now, each color $i$ for $i=1,\ldots,k-1$ has the property that its essential component cannot intersect both $A$ and $B$. So either $A$ or $B$ have the property that they intersect
$t \le \lfloor (k-1)/2 \rfloor$ essential parts of the colors $1,\ldots,k-1$. We will consider the case that $B$ intersects $t \le \lfloor (k-1)/2 \rfloor$ essential parts (the proof for $A$ is
similar and  in fact easier since $A$ is larger than $B$).
Without loss of generality, the essential parts of $G_1,\ldots,G_t$ intersect $B$ and the essential parts of $G_{t+1},\ldots,G_{k-1}$ do not intersect $B$.

So, the complete graph induced on $B$ (namely, $K_{a-1}$) has the property that it is composed of $t$ spanning bipartite graphs $H_1,\ldots,H_t$, where $H_i$ is the subgraph of $G_i$ induced on $B$,
and $k-t$ matchings (these matchings are from the non-essential parts of the colors $t+1,\ldots,k-1$ whose essential parts do not intersect $B$).
Since $H_i$ is bipartite it has a bi-partition $L_i \cup R_i$. We associate with each vertex $v$ of $B$ a binary vector of length $t$
where the $i$'th coordinate is $1$ if $v \in R_i$ and $0$ if $v \in L_i$. Altogether there are $2^t$ possible vectors, distributed over the $a-1$ vertices of  $B$.

So, there is a subset $B' \subset B$ of size at least $(a-1)/2^t$  such that any two vertices of $B$ are associated with the same vector.  Now, consider $u,v \in B'$. The edge connecting them cannot
be colored with any of  the colors $1,\ldots,t$, since they received the same vector.  Hence, the edge connecting them must be from one of the non-essential parts of the colors $t+1,\ldots,k-1$. But
since the non-essential parts of these $k-1-t$ colors are a union of $k-1-t$ matchings, in order to get a contradiction it suffices to prove that $|B'|-1 > k-1-t$ or, if $|B'|$ is odd, it suffices
to prove that $|B'|-1 \ge k-1-t$.

So we are left with the issue of verifying that $\lceil (a-1)/2^t \rceil > k-t$ or, if $\lceil (a-1)/2^t \rceil$ is odd, it suffices to show that $\lceil (a-1)/2^t \rceil \ge k-t$.
Using the fact that $a-1 \ge \lceil \sqrt{(2^{k-1}-1)(2^k-k-1)/k} \rceil$ and that $t \le \lfloor (k-1)/2 \rfloor$ this amounts to verifying the following inequality:
$$
\left \lceil \frac{\lceil \sqrt{(2^{k-1}-1)(2^k-k-1)/k} \rceil}{2^{\lfloor (k-1)/2 \rfloor}} \right\rceil > k- \lfloor \frac{k-1}{2} \rfloor
$$
or, if the left hand side is odd, it suffices to prove a weak inequality.

Notice that for $k \ge 10$ the strong inequality is true even if we remove the ceilings in the l.h.s. and remove the floor in the denominator of the l.h.s.
For $k=4,\ldots,9$ we verify explicitly:

For $k=4,5$ we have the inequality $3 \ge 3$ which is true (here we use the fact the the l.h.s. is odd so the weak inequality suffices).
For $k=6$ we have the inequality $\lceil 18/4 \rceil > 6-2$.
For $k=7$ we have the inequality $\lceil 33/8 \rceil > 7-3$.
For $k=8$ we have the inequality $\lceil 63/8 \rceil > 8-3$.
For $k=9$ we have the inequality $\lceil  120/16 \rceil > 9-4$.
Hence the lemma holds for all $k \ge 4$.
\end{proof}

\noindent
We now turn to prove the lower bound in Theorem \ref{t:2}.
\begin{lemma}\label{l:2-lower}
$M_k(3) \ge \frac{3}{4}2^k+2$ for $k \geq 3$.
\end{lemma}
\begin{proof}
We have already shown in Figure \ref{diag1} that there is an edge coloring of $K_7$ with $3$ colors without an mdm-path of order $3$.
Figure \ref{diag2} gives constructions of edge colorings of $K_6$ and $K_5$ with $3$ colors without an mdm-path of order $3$ (recall: we cannot just use
the coloring for $K_7$ to deduce this for $K_6$ and $K_5$ as the degree-monotone property is not hereditary).
Hence by Remark 1, using $k=3$, $m=3$ and $t=7$, we have that we can $4$-color $K_{13},K_{12},K_{11}$ with no mdm-path of order $3$,
we can $5$-color $K_{25},K_{24},K_{23}$ with no mdm-path of order $3$, and the process continues so that we can $k$-color
$K_{(3/4)2^k+1}, K_{(3/4)2^k}, K_{(3/4)2^k-1}$ with no mdm-path of order $3$, so in particular $M_k(3) \ge  \frac{3}{4}2^k+2$ for $k \ge 3$.
\end{proof}

\noindent
Theorem \ref{t:2} now follows from Corollary \ref{col:1}, Lemma \ref{l:mk3} and Lemma \ref{l:2-lower}. \qed

\begin{figure}[h]
\centering
\includegraphics{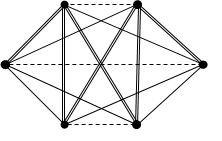} \hspace{2cm} \includegraphics{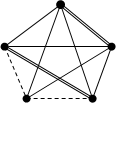}
\caption{$3$-edge-colorings of $K_6$ and $K_5$ with no mdm-path of order $3$.} \label{diag2}
\end{figure}

\section{Proof of Theorem \ref{t:3}}
We prove that $M(3,m)=2(m-1)$. Namely, for $n \ge 2(m-1)$, in a $2$-coloring of the edges of $K_n$ there is either a degree-monotone path of order $3$ in color $2$ or a degree-monotone path of
ordered $m$ in color $1$, and that this is tight.

Starting with tightness, observe that it is obtained by coloring $G  = K_{m-1,m-2}$ with color $2$ and coloring its complement with color $1$.

So suppose $n \ge 2(m-1)$ and $K_n$ is edge-colored using two colors $1$ and $2$.
If $n > 2(m-1)$ the result follows from the upper bound in Theorem \ref{t:1}. So we assume $n=2(m-1)$.
Let $G_j$ be the spanning graph of $K_{2(m-1)}$ induced by the edges with color $j$, $j=1,2$. We will show that either $G_1$ has a has degree monotone path of order $m$ or
$G_2$ has a degree monotone path of order $3$.

We know by Lemma \ref{l:1} and the paragraph following it that $\chi(G_1)=m-1$ and $\chi(G_2)=2$, and further, $G_1$ can be split into two cliques $A$ and $B$ such that $|A|=|B|=m-1$,
with possible edges between $A$ and $B$.
Now if we order the vertices in both $A$ and $B$ in non-decreasing order of their degrees in $G_1$,
say $A=\{v_1,v_2,\ldots,v_{m-1}\}$ and $B=\{u_1,u_2,\ldots,u_{m-1}\}$, then in $G_1$, $v_i$ is not adjacent to
$u_i$ for otherwise there will be a degree monotone path of order $m$, either $v_1,\ldots,v_i,u_i,\ldots,u_{m-1}$ or $u_1,\ldots,u_i,v_i,\ldots,v_{m-1}$.

Hence in $G_2$, $v_i$ and $u_i$ are adjacent and $deg(v_i) \geq deg(v_{i+1})$, as well as $deg(u_i) \geq deg(u_{i+1})$ in $G_2$.
Assume without loss of generality, that $deg(v_1)$ is maximal in $G_2$. If $deg(v_1) = 1$, then $G_2$ is  a matching --- but then clearly there is a monotone path of order $m$ in $G_1$. 
Thus, $deg(v_1)  \geq 2$.  We may assume $deg(v_1) > deg(u_1)$ for otherwise $deg(v_1) = deg(u_1) \geq 2$ will force a monotone path of order $3$.
 
Now all the neighbors of $v_1$  are in $B$ and hence have degrees less than $deg(v_1)$. Consider $N(N(v_1))$ which are all in $A$ and each vertex in $N(N(v_1))$ must have a degree strictly
greater than the degrees of its neighbor in $N(N(N(v_1)))$ in $B$ (otherwise there is monotone degree path of order $3$).   
Continuing this way, then either $G_2$ is connected and  illusive or contains an illusive component with balanced sides $A^*$ and $B^*$, $|A^*| = |B^*|$,  guaranteed by the matching
$(v_i u_i)$ for $i = 1,\ldots,(m-1)$. But illusive graphs are impossible, hence $G_2$ contains a degree-monotone path of order $3$.
\qed

\section{Some open problems}

As mentioned in the introduction, having a degree-monotone path of a certain order is not a hereditary property.
Hence the following problem seems of interest.
Let ${\cal N}_k(m)$ be the set of all positive integers such that $n \in {\cal N}_k(m)$ if and only if in every $k$-coloring of the edges of $K_n$ there is a monochromatic mdm-path of order $m$.
\begin{problem}
Is it true that for all $k$ and $m$, ${\cal N}_k(m)$ has no gaps.
\end{problem}
Recall that the proof in Section 3 gives that ${\cal N}_2(3)$, ${\cal N}_3(3)$, ${\cal N}_4(3)$ have no gaps. For ${\cal  N}_4(3)$ this follows since Theorem \ref{t:2} gives $14 \le M_4(3) \le 15$
and since the construction in Lemma \ref{l:2-lower} together with some small case analysis can be used to prove that for all $n \le 13$, there are $4$-edge colorings of $K_n$ with no mdm-path
of order $3$. We can also show (see below) that ${\cal N}_2(4)$ has no gaps.

Theorem \ref{t:2} asserts that
$$
\frac{3}{4} \le \liminf_{k \rightarrow \infty} \frac{M_k(3)}{2^k} \le \limsup_{k \rightarrow \infty} \frac{M_k(3)}{2^k} \le 1\;.
$$
\begin{problem}
Determine if $\lim_{k \rightarrow \infty} \frac{M_k(3)}{2^k}$ exists and determine it.
\end{problem}

The diagonal case with two colors, namely $M_2(m)$, may be the most accessible. By Theorem \ref{t:1} we know that $M_2(m) \le (m-1)^2+1$.
\begin{conjecture}
For every constant $C$, if $m$ is sufficiently large, then $M_2(m) \le (m-1)^2-C$.
\end{conjecture}
Recall that we know that $M_2(3)=4$ and we have also verified (using a computer) that $M_2(4) = 7$.
For the latter we needed to verify that all $2$-edge colorings of $K_7,K_8,K_9$ have an mdm-path of order $4$
(recall that Theorem \ref{t:1} guarantees that $7 \le M_2(4) \le 10$).

\bibliographystyle{plain}

\bibliography{dmp}

\end{document}